\documentclass[11pt,reqno]{amsart}
\usepackage[utf8]{inputenc} 
\usepackage[T1]{fontenc}    
\usepackage{hyperref}
\usepackage{amsfonts}       
\usepackage{microtype}      

\usepackage{amscd,amssymb,amsmath,amsthm}
\usepackage[arrow,matrix]{xy}
\usepackage{graphicx}
\usepackage{multicol}
\usepackage{subfigure}
\usepackage{epstopdf}
\usepackage{color}
\usepackage{soul}
\newtheorem{thm}{Theorem}
\newtheorem{defn}{Definition}
\newtheorem{lemma}{Lemma}
\newtheorem{pro}{Proposition}
\newtheorem{rk}{Remark}

\numberwithin{equation}{section} 

\begin{document}

\vspace{0.5in}
\renewcommand{\bf}{\bfseries}
\renewcommand{\sc}{\scshape}
\vspace{0.5in}

\title[Stability and Bifurcation Analysis of a Phytoplankton-Zooplankton Model]{Stability and Bifurcation Analysis of a Phytoplankton-Zooplankton Model with Linear Functional Responses}

\author{S.K. Shoyimardonov}

\address{S.K. Shoyimardonov$^{a,b}$ \begin{itemize}
\item[$^a$] V.I.Romanovskiy Institute of Mathematics, 9, University str.,
Tashkent, 100174, Uzbekistan;
\item[$^b$] National University of Uzbekistan,  4, University str., 100174, Tashkent, Uzbekistan.
\end{itemize}}
\email{shoyimardonov@inbox.ru}





\keywords{Lyapunov function, LaSalle Invariance Principle, phytoplankton, zooplankton, bifurcation, Neimark-Sacker}

\subjclass[2010]{34D20 (92D25)}

\begin{abstract}
In this paper, the dynamics of a phytoplankton-zooplankton system with linear functional responses are examined. For the continuous-time model, the global asymptotic stability of the fixed points is demonstrated by constructing Lyapunov functions. For the discrete version of the model, both local and global dynamics are investigated using LaSalle's Invariance Principle. Furthermore, the occurrence of a Neimark-Sacker bifurcation at the positive fixed point is established, and it is proved that the resulting invariant closed curve is attracting.
\end{abstract}

\maketitle

\section{Introduction}

The interaction between phytoplankton and zooplankton plays a fundamental role in aquatic ecosystems, shaping nutrient cycling, energy transfer, and the overall health of marine and freshwater environments. Mathematical modeling of these interactions offers powerful tools for understanding population dynamics and assessing the stability of such ecosystems.

 Historically, some of the earliest mathematical representations of predator-prey interactions were developed by Lotka and Volterra in the 1920s. Their continuous-time models laid the foundation for modern ecological modeling. Over time, these frameworks have been extended to incorporate more realistic ecological features, including functional responses (describing how predator consumption depends on prey availability), intraspecific competition, and factors such as environmental variability and time delays. These models have found broad applications in fields ranging from ecology and epidemiology to physics and wildlife conservation. They not only elucidate predator-prey dynamics but also reveal how changes in environmental conditions, resource abundance, or anthropogenic factors can affect ecosystem stability.

While continuous-time models assume smooth evolution over time, discrete-time models are particularly valuable when biological interactions or data collection occur at distinct intervals. Discrete formulations naturally capture periodic processes such as seasonal breeding, harvesting, and sampling, making them highly applicable in ecological settings  \cite{All, May}. Moreover, discrete-time systems often display a richer spectrum of dynamical behaviors, including period-doubling bifurcations, Neimark-Sacker bifurcations, and chaos--which may not arise in continuous-time analogues \cite{De, El}. These phenomena provide important insight into complex dynamics such as oscillations and regime shifts.

In this work, we analyze both the discrete and continuous-time versions of a phytoplankton-zooplankton system, modeled by the following set of differential equations:

\begin{equation}\label{chat}
\left\{\begin{aligned}
&\frac{dP}{dt}=bP(1-\frac{P}{k})-\alpha f(P)Z,\\
&\frac{dZ}{dt}=\beta f(P)Z-rZ-\theta g(P)Z,
\end{aligned}\right.
\end{equation}
where $P$ represents the density of the phytoplankton population, and $Z$ denotes the density of the zooplankton population. The parameter $b$ is the intrinsic growth rate of phytoplankton, and $k$ represents the carrying capacity of the environment. The term
$f(P)$ characterizes the functional response of zooplankton to phytoplankton availability, while $\alpha$ and $\beta$
 are the phytoplankton consumption rate and the conversion efficiency of zooplankton, respectively.  The zooplankton population dynamics are influenced by energy gains from consuming phytoplankton ($\beta f(P)$), natural mortality ($rZ$). An additional feature of this model is the inclusion of $g(P)$, which characterizes the distribution of toxin substances produced by phytoplankton. The parameter $\theta$ represents the rate at which phytoplankton release toxins, which can negatively impact zooplankton survival. This toxin-mediated interaction introduces complexity to the predator-prey dynamics and has significant implications for the stability and resilience of the ecosystem.

Phytoplankton-zooplankton interactions have been extensively studied by researchers using various modeling approaches (see \cite{Chatt, Chen, Hen, Hong, Mac, RSH, RSHV, Sajan, SH, Tian} and references therein). The model considered in this study was initially introduced by the authors in \cite{Chatt}. In \cite{SH}, the discrete-time version of the model was analyzed with a Holling type II predator functional response and a linear prey functional response. The study demonstrated the occurrence of a Neimark-Sacker bifurcation at a positive fixed point. Similarly, in \cite{Chen}, the continuous-time version of the model with a Holling type II predator functional response was investigated. It was shown that the local stability of the positive equilibrium implies global stability if there is a unique positive equilibrium. However, when multiple positive equilibria exist, the local stability of the equilibrium with a small phytoplankton population density leads to bistability. These results were extended to include a Holling type III predator functional response under specific conditions.

In this paper, we investigate the local and global stability of the fixed points and examine the occurrence of bifurcations. The paper is organized as follows: Section 2 explores the global dynamics of the continuous-time model through the construction of a Lyapunov function. In Section 3, we analyze the discrete-time version of the model, applying LaSalle’s Invariance Principle to establish the stability of its fixed points. In Section 4, we demonstrate the occurrence of a Neimark-Sacker bifurcation at the unique positive fixed point and show that the resulting invariant closed curve is attracting.

Consider the model (\ref{chat}) by choosing $f(P)=g(P)=P$  and simplifying as
$$\overline{t}=bt, \, \overline{x}=\frac{P}{k}, \, \overline{y}=\frac{\alpha Z}{b}, \, \overline{\beta}=\frac{\beta k}{b}, \, \overline{r}=\frac{r}{b}, \, \overline{\theta}=\frac{\theta k}{b}.$$

Then by dropping the overline sign at time $t\geq0$ we get
\begin{equation}\label{chenn}
\left\{\begin{aligned}
&\dot{x}=x(1-x)-xy\\
&\dot{y}=(\beta-\theta) xy-ry\\
\end{aligned}\right.
\end{equation}

Let's denote $\gamma=\beta-\theta$ and we have the following continuous-time model:
\begin{equation}\label{contin}
\left\{\begin{aligned}
&\dot{x}=x(1-x)-xy\\
&\dot{y}=\gamma xy-ry\\
\end{aligned}\right.
\end{equation}
and corresponding discrete-time model is:

\begin{equation}\label{h12}
V:
\begin{cases}
x^{(1)}=x(2-x)-xy\\[2mm]
y^{(1)}=\gamma xy+(1-r)y.
\end{cases}
\end{equation}
In the sequel of the paper we assume that parameters $ r, \gamma-$ are-positive numbers.
It is obvious that, system (\ref{h12}) always has two nonnegative equilibria $E_0=(0; 0)$ and $E_1=(1; 0)$. From $x^{(1)}=x,$ $y^{(1)}=y$ we get the following positive fixed point $E_2=(\frac{r}{\gamma}; \frac{\gamma-r}{\gamma}),$ where $\gamma>r.$
We note that the equilibrium points of the continuous model (\ref{contin}) coincide with the fixed points of the discrete model
(\ref{h12}).

\begin{rk}
In classical Lotka–Volterra systems, the prey population is typically assumed to grow exponentially. In contrast, system~(\ref{contin}) incorporates logistic growth for the prey, introducing a carrying capacity and adding a layer of nonlinear complexity to the dynamics. Earlier studies largely focused on models with exponential prey growth or on systems with more general nonlinear predator responses, such as those described by Holling-type functions. As a result, systems combining logistic prey growth with linear predation--though ecologically reasonable--were not central to classical ecological modeling and received comparatively less attention.
\end{rk}

\section{Stability Analysis of Continuous Model (\ref{contin})}

\begin{thm}\label{thm1} If $0<\gamma\leq r,$ then the equilibrium point $E_1=(1,0)$ of the continuous model (\ref{contin}) is globally asymptotically stable.
\end{thm}
\begin{proof} Note that in this case there is no positive fixed point. We determine the following function as a Lyapunov function:

$$L(x,y)=x-\ln x+\frac{y}{\gamma}-1.$$
Obviously, $L(1,0)=0.$ First, we show that $L(x,y)>0$ for any $(x,y)\in \mathbb{R}^2_{+}\setminus{E_1}.$ We analyze the function $f(x)=x-\ln x.$ Its derivative $f'(x)=1-\frac{1}{x}$ and $f'(x)>0$ when $x>1,$  $f'(x)<0$ when $0<x<1,$ $f'(x)=0$ at $x=1.$ Thus, the function $f(x)$ has a minimum value $f(1)=1$ at $x=1.$ So, $L(x,y)>0.$
Now, we prove that $\dot{L}\leq0:$
\[
\dot{L}=(x-1)(1-x-y)+\frac{y(\gamma x-r)}{\gamma}=-(1-x)^2+y\left(1-\frac{r}{\gamma}\right).
\]
Since $0<\gamma\leq r,$ we have  $\dot{L}<0 \ \ \forall (x,y)\in \mathbb{R}^2_{+}\setminus{E_1}.$
In addition, $L(\mathbf{x})\rightarrow\infty$ as $\|\mathbf{x}\|\rightarrow\infty,$ i.e., $L(x,y)>0$ is radially unbounded, so the fixed point $(1,0)$ is globally asymptotically stable.
\end{proof}

Let $\hat{x}=\frac{r}{\gamma}, \hat{y}=1-\frac{r}{\gamma}$ be the coordinates of the equilibrium point $E_2.$

\begin{thm}\label{thm2} If $\gamma>r,$ then the equilibrium point $E_2$ of the continuous model (\ref{contin}) is globally asymptotically
stable.
\end{thm}

\begin{proof} We define the following function
$$H(x,y) = \hat{x}\ln x-x+ \frac{\hat{y}\ln y-y}{\gamma}.$$
Using this function we construct the following Lyapunov function:
$$ L(x,y)= H(\hat{x},\hat{y})-H(x,y).$$
Obviously, $L(\hat{x},\hat{y})=0.$ Now we check that $\dot{L}\leq0.$ For this we use that
\begin{equation*}\label{cont}
\left\{\begin{aligned}
&\hat{x}+\hat{y}=1\\
&\gamma\hat{x}=r.\\
\end{aligned}\right.
\end{equation*}
Then
$$\dot{L}=(x-\hat{x})(1-x-y)+(x-\hat{x})(y-\hat{y})=-(x-\hat{x})^2<0, \ \  \forall (x,y)\in \mathbb{R}^2_{+}\setminus{E_2}.$$
Now, we demonstrate that $L(x,y)$ is positive definite.
Consider the function
$$u(a,b)=a\ln a-a\ln b+b-a.$$
We prove that $u(a,b)\geq0$ for any $a>0, b>0.$ If $a=b$ then $u(a,b)=0.$

Let $a>b.$ We then apply the Mean Value Theorem to the function $f(x)=a\ln x$ on the interval $[b,a].$ Thus,
$$\frac{f(a)-f(b)}{a-b}=f'(c),$$ where
$$b<c<a \Rightarrow  \frac{1}{a}<\frac{1}{c}<\frac{1}{b} \Rightarrow 1<\frac{a}{c}<\frac{a}{b} \Rightarrow 1<f'(c)<\frac{a}{b}.$$
So, we have that $u(a,b)>0.$ The case $a<b$ can be proven in the same way. We note that $L(x,y)=u(\hat{x},x)+\gamma^{-1}u(\hat{y},y)$ and from this we get the proof that $L(x,y)$ is positive defined. Moreover, it can be easily seen that $L(x,y)$ is radially unbounded, so the equilibrium point $E_2$ is globally asymptotically stable. \end{proof}

\section{Stability Analysis of Discrete Model (\ref{h12})}

\begin{defn}\label{def1} Let $E(x,y)$ be a fixed point of the operator $F:\mathbb{R}^{2}\rightarrow\mathbb{R}^{2}$ and $\lambda_1, \lambda_2$ are eigenvalues of the Jacobian matrix $J=J_{F}$ at the point $E(x,y).$

(i) If $|\lambda_1|<1$ and $|\lambda_2|<1$ then the fixed point $E(x,y)$ is called an attractive or sink;

(ii) If $|\lambda_1|>1$ and $|\lambda_2|>1$ then the fixed point $E(x,y)$ is called  repelling or source;

(iii) If $|\lambda_1|<1$ and $|\lambda_2|>1$ (or $|\lambda_1|>1$ and $|\lambda_2|<1$) then the fixed point $E(x,y)$ is called  saddle;

(iv) If either $|\lambda_1|=1$ or $|\lambda_2|=1$ then the fixed point $E(x,y)$ is called to be  non-hyperbolic;
\end{defn}

Before analyze the fixed points we give the following useful lemma.
\begin{lemma}[Lemma 2.1, \cite{Cheng}]\label{lem1} Let $F(\lambda)=\lambda^2+B\lambda+C,$ where $B$ and $C$ are two real constants. Suppose $\lambda_1$ and $\lambda_2$ are two roots of $F(\lambda)=0.$ Then the following statements hold.
\begin{itemize}
 \item[(i)]  If $F(1)>0$ then
{\begin{itemize}
\item[(i.1)]
$|\lambda_1|<1$ and $|\lambda_2|<1$ if and only if $F(-1)>0$ and $C<1;$
\item[(i.2)]
 $\lambda_1=-1$ and $\lambda_2\neq-1$ if and only if $F(-1)=0$ and $B\neq2;$
\item[(i.3)]
$|\lambda_1|<1$ and $|\lambda_2|>1$ if and only if $F(-1)<0;$
\item[(i.4)]
$|\lambda_1|>1$ and $|\lambda_2|>1$ if and only if $F(-1)>0$ and $C>1;$
\item[(i.5)]
$\lambda_1$ and $\lambda_2$ are a pair of conjugate complex roots and $|\lambda_1|\!=\!|\lambda_2|\!=\!1$ if and only
       if $-2<B<2$ and $C=1;$
\item[(i.6)]
 $\lambda_1=\lambda_2=-1$ if and only if $F(-1)=0$ and $B=2.$
\end{itemize}}
\item[(ii)]  If $F(1)=0,$ namely, 1 is one root of $F(\lambda)=0,$ then the other root $\lambda$ satisfies
$|\lambda|=(<,>)1$ if and only if $|C|=(<,>)1.$
\item[(iii)]  If $F(1)<0,$ then $F(\lambda)=0$ has one root lying in $(1;\infty).$ Moreover,
\begin{itemize}
\item[(iii.1)]
 the other root $\lambda$ satisfies $\lambda<(=)-1$ if and only if $F(-1)<(=)0;$
\item[(iii.2)]
 the other root $\lambda$ satisfies $-1<\lambda<1$ if and only if $F(-1)>0.$
 \end{itemize}
\end{itemize}
\end{lemma}

\begin{pro}\label{prop1} For the fixed points $E_0=(0,0),$ $E_1=(1,0)$ and $E_2=(\frac{r}{\gamma} ,\frac{\gamma-r}{\gamma}),$ of (\ref{h12}),  the following statements hold true:

(i)
\[
E_{0}=\left\{\begin{array}{lll}
{\rm nonhyperbolic}, \ \ {\rm if} \ \  r=2\\[2mm]
{\rm saddle}, \ \ \ \ {\rm if} \ \  0<r<2\\[2mm]
{\rm repelling}, \ \ \ \  {\rm if}  \ \ r>2,
\end{array}\right.
\]

(ii)
\[
E_{1}=\left\{\begin{array}{lll}
{\rm nonhyperbolic}, \ \ {\rm if} \ \   \gamma=r \ \ {\rm or} \ \  \gamma=r-2\\[2mm]
{\rm attractive}, \ \ \ \ {\rm if} \ \  r-2<\gamma<r\\[2mm]
{\rm saddle}, \ \ \ \  {\rm if}  \ \  otherwise,
\end{array}\right.
\]

(iii)
\[
E_2=\left\{\begin{array}{lll}
{\rm nonhyperbolic}, \ \ {\rm if} \ \  \gamma=r+1 \\[2mm]
{\rm attractive}, \ \ {\rm if} \ \  r<\gamma<1+r\\[2mm]
{\rm repelling}, \ \  {\rm if} \ \ \gamma>r+1.
\end{array}\right.
\]
\end{pro}

\begin{proof} (i) The Jacobian of the operator (\ref{h12}) is
\begin{equation}\label{jac}
J(x,y)=\begin{bmatrix}
2-2x-y  & -x\\
\gamma y & \gamma x+1-r
\end{bmatrix}.
\end{equation}
Then $J(0,0)=\begin{bmatrix}
2 & 0\\
0 & 1-r
\end{bmatrix}$  and eigenvalues are $2$ and $1-r.$ From this, for $E_0$ we can take the proof easily.

(ii) The Jacobian at the fixed point $E_1$ is $J(1,0)=\begin{bmatrix}
0 & -1\\
0 & \gamma +1-r
\end{bmatrix}$ and the eigenvalues are $\lambda_1=0, \lambda_2=\gamma+1-r.$ By solving $|\lambda_2|<1$ we get the condition $r-2<\gamma<r.$

(iii) Now consider the Jacobian at the positive fixed point using $E_2=(\frac{r}{\gamma},\frac{\gamma-r}{\gamma})$:

\[
J(E_2)=\begin{bmatrix}
1-\frac{r}{\gamma} & -\frac{r}{\gamma}\\
\gamma-r & 1
\end{bmatrix}.
\]

The characteristic polynomial is
\[
F(\lambda)=\lambda^{2}-\left(2-\frac{r}{\gamma}\right)\lambda+\frac{(\gamma-r)(1+r)}{\gamma}.
\]

It can be easily seen that $F(1)>0, F(-1)>0.$ The free coefficient of $F(\lambda)$ is

\[
C=\frac{(\gamma-r)(1+r)}{\gamma}.
\]
By solving the inequalities $C<1, C>1$ and  according to Lemma \ref{lem1} we obtain the proof.
\end{proof}

\begin{pro} The following sets
\begin{equation}
\mathcal{X}=\{(x,y)\in \mathbb{R}_+^2: 0\leq x\leq2, y=0\}, \ \ \mathcal{Y}=\{(x,y)\in \mathbb{R}_+^2: x=0, y\geq0\},
\end{equation}
are invariant w.r.t. operator (\ref{h12}). In this operator, the condition $r\leq 1$ must be satisfied for invariance of $\mathcal{Y}.$
\end{pro}

\begin{proof} It is clear that $y^{(1)}=0$ when $y=0.$ If $y=0$ then $x^{(1)}=x(2-x)\geq0$ and the maximum value in [0,2] of the quadratic function $f(x)=x(2-x)$ is 1. Hence, the set $\mathcal{X}$ is an invariant. If $x=0$ then $y^{(1)}=(1-r)y\geq0$ and $\mathcal{Y}$ also invariant set. The proof is completed.
\end{proof}

\subsubsection{Dynamics on invariant sets $\mathcal{X}$ and $\mathcal{Y}$}
\hfill

\emph{Case:} $\mathcal{X}.$ In this case the restriction is
$$x^{(1)}=x(2-x)=f(x),  \ \ x\in[0,2].$$
Note that $f(x)$ has two fixed points $x=0$ and $x=1.$ Since $f'(x)=2-2x$ the fixed point 0 is repelling and 1 is an attracting fixed point. Moreover, the equation
$$\frac{f^2(x)-x}{f(x)-x}=0 \Rightarrow \frac{x^4-4x^3+6x^2-3x}{x^2-x}=0 \Rightarrow x^2-3x+3=0 $$
has no real solution. Consequently, the function $f$ does not have two periodic points and, by Sarkovskii's theorem, there is no any periodic point (except fixed point).

Let's study the dynamics of $f(x):$ If $x\in(0,1)$ then $0<x<f(x)<1.$ So
\[
0<f(x)<f^2(x)<1 \Rightarrow  \cdots \Rightarrow 0<f^n(x)<f^{n+1}(x)<1.
\]
Thus, the sequence $f^n(x)$ is monotone increasing and bounded from above, which has the limit. Since 1 is a unique fixed point in $(0,1],$ we have that
\[
\lim_{n\to\infty}f^{n}(x)=1.
\]
If $x$ lies in the interval (1,2) then $f(x)$ lies in (0,1), so that the previous argument implies $f^n(x)=f^{n-1}(f(x))\rightarrow1$  as $n\rightarrow\infty.$
Hence, 1 is globally attractive fixed point. (similar to quadratic family, \cite{De}, page 32 ).

\emph{Case:} $\mathcal{Y}.$ In this case the restriction is $$y^{(1)}=(1-r)y,  \ \ 0<r\leq1$$
and $y^{(n)}=(1-r)^ny\rightarrow0$ as $n\rightarrow\infty.$ Thus, the dynamics on $\mathcal{X}$ and $\mathcal{Y}$ is clear.

Denote
\[
S_1=\left\{(x,y)\in R^2_{+}: 0<x\leq 3-2\sqrt{2}, \ \ \frac{(1-\sqrt{x})^2}{2}\leq y\leq \frac{(1+\sqrt{x})^2}{2} \right\}
\]

\[
S_2=\left\{(x,y)\in R^2_{+}: 3-2\sqrt{2}<x<1, \ \ x< y\leq \frac{(1+\sqrt{x})^2}{2}\right\}
\]
and $S=S_1\cup S_2$ be the set represented in Fig. \ref{fig1}.

 \begin{figure}[h!]
  \centering
  \includegraphics[width=5cm]{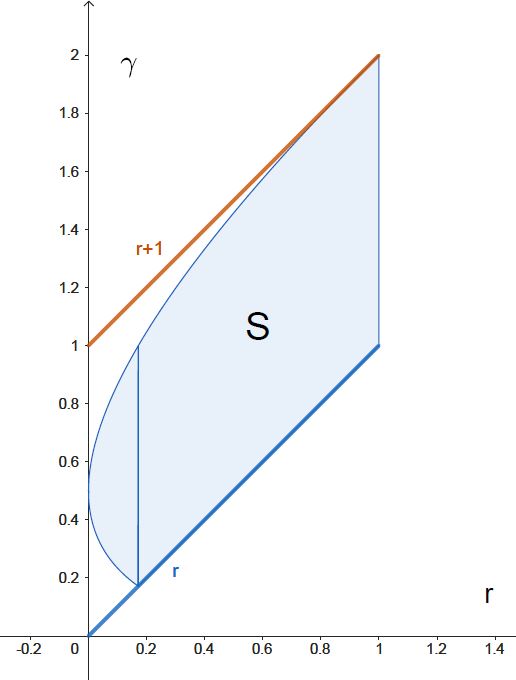}\\
  \caption{The parameter set $S$}\label{fig1}
\end{figure}

\begin{lemma}\label{lem2} Let $0<r\leq1.$  If $0<\gamma\leq r$ or $(r,\gamma)\in S$ then the following set
\[
\mathcal{M}=\{(x,y)\in \mathbb{R}_+^2: 0\leq x\leq1, \ \ 0\leq y\leq2-x\}
\]
is an invariant w.r.t operator (\ref{h12}).
\end{lemma}

\begin{proof} Let $(x,y)\in \mathcal{M}.$ Since $0<r\leq1, \gamma>0$ and  $0\leq y\leq2-x$ we have that $x^{(1)}=x(2-x-y)\geq0$ and   $y^{(1)}=y(\gamma x+1-r)\geq0.$ Moreover, the quadratic function $x(2-x)$ achieves its maximum value of 1, so $x^{(1)}=x(2-x-y)\leq x(2-x)\leq1,$ thus, $0\leq x^{(1)}\leq1.$ Now, we have to prove that $y^{(1)}\leq2-x^{(1)}.$

\emph{Case-1}.  Let $0<\gamma\leq r.$ Then $y^{(1)}=y(\gamma x+1-r)\leq y(rx+1-r)\leq y(r+1-r)=y.$ In addition, $x^{(1)}\geq x$ when $y\leq1-x$ and $x^{(1)}\leq x$ when $y\geq1-x.$ Since  $y^{(1)}\leq y,$ the case $y^{(1)}>2-x^{(1)}$ is impossible when $y\leq1-x.$ We only need to consider the case where $y\geq1-x.$ From $x^{(1)}\leq x$ and $y^{(1)}\leq y,$ we get that $y^{(1)}\leq y\leq2-x\leq2-x^{(1)}.$ Thus,  $y^{(1)}\leq 2-x^{(1)}$ when $0<\gamma\leq r.$

\emph{Case-2}. Let $\gamma> r.$ Now, we determine the conditions on the parameters $r$ and $\gamma$ such that $y^{(1)}\leq 2-x^{(1)}$ is satisfied.
\[
y^{(1)}+x^{(1)}-2\leq0 \ \ \Rightarrow \ \ 2x-x^2-2+y(\gamma x-x+1-r)\leq0,
\]

\[
y(x-\gamma x+r-1)\geq 2x-x^2-2.
\]
It is fact that $2x-x^2-2<0.$ Moreover, if $\gamma\geq1$ then $x-\gamma x+r-1<0$ and we have
 \[
 y\leq \frac{2x-x^2-2}{x-\gamma x+r-1}.
 \]
 If $\gamma<1$ then from $\gamma>r$ we get $\frac{1-r}{1-\gamma}>1,$ so $x<\frac{1-r}{1-\gamma},$ i.e., $x-\gamma x+r-1<0,$ again we have
  \[
 y\leq \frac{2x-x^2-2}{x-\gamma x+r-1}.
 \]
 Now we consider the equation
 \begin{equation}\label{quad}
 \frac{2x-x^2-2}{x-\gamma x+r-1}=2-x
 \end{equation}
  for which has at most one solution (Fig \ref{fig2}). From (\ref{quad}) we obtain
  \[
  \gamma x^2-(2\gamma+r-1)x+2r=0.
  \]

  \begin{figure}[h!]
  \centering
  \includegraphics[width=5cm]{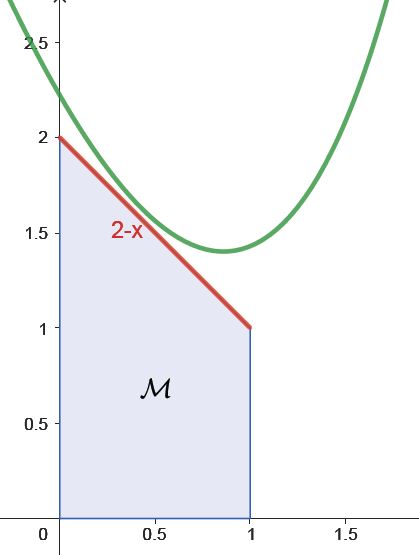}\\
  \caption{The green and red lines do not intersect, ensuring the invariance of $\mathcal{M}$}\label{fig2}
\end{figure}
 The discriminant of this equation must be $\leq0.$ So we have
 \[
 4\gamma^2-4\gamma r-4\gamma+(r-1)^2\leq0.
  \]
  Solving this inequality we get that $(r,\gamma)\in S.$ The lemma has been proven.
\end{proof}

The following theorem provides insight into the dynamics of the operator (\ref{h12}).

\begin{thm}\label{thm3} Let $0<r\leq1,$ and let the initial point $(x^0,y^0)$ be chosen from the set $\mathcal{M}$ with $x^0>0$.

(i) If $0<\gamma\leq r$ then
\[
\lim_{n\to\infty}V^{n}(x^0,y^0)=E_1.
\]

(ii) If $(r,\gamma)\in S$  then
\[
\lim_{n\to\infty}V^{n}(x^0,y^0)=E_2.
\]

\end{thm}

\begin{proof} Under conditions to parameters we note that the set $\mathcal{M}$ is an invariant. Recall that the quadratic function $x(2-x)$ achieves its maximum value of 1 on the interval [0,2].  This implies that if
$1<x^0\leq2$, then after one iteration, $0<x^{(1)}\leq1$. Therefore, it suffices to consider the first coordinate of the initial point within the interval [0,1].

(i). Let $0<\gamma\leq r.$  Note that, in this case, there is no positive fixed point, the fixed point $E_1=(1,0)$ is an attracting, so there exists a neighbourhood $U(E_1)$ such that the trajectory starting from  it converges to $E_1.$ Moreover, as demonstrated in the proof of Lemma \ref{lem2}, the sequence $y^{(n)}$ is always decreasing and bounded below by zero, which ensures that it has a limit. In addition, the sequence $x^{(n)}$  decreases when it is above the line
$y=1-x$ and increases when it is below that line.  Thus, the trajectory starting from the initial point $(x^0,y^0)\in \mathcal{M}$ falls into the neighbourhood $U(E_1)$ after some finite steps and then it converges to the fixed point $E_1=(1,0).$ To prove the statement, we employ LaSalle's Invariance Principle.

Define the following continuous function for any $c\geq\frac{1}{r-\gamma}$:
\[
L(x,y)=1-x+cy.
\]
Then $L(x,y)\geq0$ and $L(1,0)=0.$ Consider

\[
\triangle L=L(x^{(1)},y^{(1)})-L(x,y)=x-x^{(1)}+c(y^{(1)}-y)=x(x-1)+y(c\gamma x+x-cr)\leq0
\]
if $c\geq\frac{1}{r-\gamma}.$ The set

\[
\mathcal{B}=\{(x,y)\in \mathcal{M}: \triangle L=0\}=\{(0,0), (1,y)\}.
\]
Assume $x > 0$ (the case $x = 0$ is studied above). Then, the largest invariant set in $\mathcal{B}$ is the fixed point $E_1 = (1, 0)$. Since the set $\mathcal{M}$ is compact, it follows from LaSalle's Invariance Principle that any trajectory starting in $\mathcal{M}$ converges to the fixed point $E_1$.

 (ii). Let $(r,\gamma)\in S.$ It can be easily shown that $r<\gamma<r+1,$ which implies the existence of a positive fixed point  $E_2=\left(\frac{r}{\gamma}, \frac{\gamma-r}{\gamma}\right).$ Moreover, this fixed point is an attractive hyperbolic point.

  Define the following piecewise function:
  \[
  L(x,y)=\left\{\begin{array}{lll}
\ln y, \ \ {\rm if} \ \   x\leq\frac{r}{\gamma}\\[2mm]
\ln \frac{1}{y}, \ \ {\rm if} \ \   x>\frac{r}{\gamma}
\end{array}\right.
  \]

  Then
  \[
  \triangle L=L(x^{(1)},y^{(1)})-L(x,y)= \left\{\begin{array}{lll}
\ln (\gamma x+1-r), \ \ {\rm if} \ \   x\leq\frac{r}{\gamma}\\[2mm]
\ln \frac{1}{\gamma x+1-r}, \ \ {\rm if} \ \   x>\frac{r}{\gamma}
\end{array}\right.
  \]

  Thus, $\triangle L\leq0$ for all $x\in[0,1].$
Next we determine the nature of the set
\[
\mathcal{B}=\{(x,y)\in \mathcal{M}: \triangle L=0\}=\{(x,y)\in \mathcal{M}: x=r/\gamma\}.
\]

Then, the largest invariant set in $\mathcal{B}$ is the fixed point $E_2$. Since the set $\mathcal{M}$ is compact, LaSalle's Invariance Principle ensures that any trajectory starting in $\mathcal{M}$ converges to the fixed point $E_2$. This completes the proof.
 \end{proof}
%
%
%
\begin{figure}[h!]
    \centering
    \subfigure[\tiny$r=0.3, \gamma=0.25$]{\includegraphics[width=0.45\textwidth]{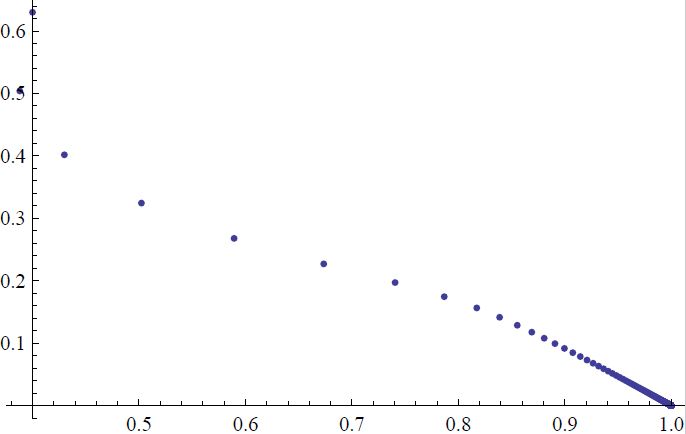}} \hspace{0.3in}
    \subfigure[\tiny $r=0.3, \gamma=0.9$]{\includegraphics[width=0.45\textwidth]{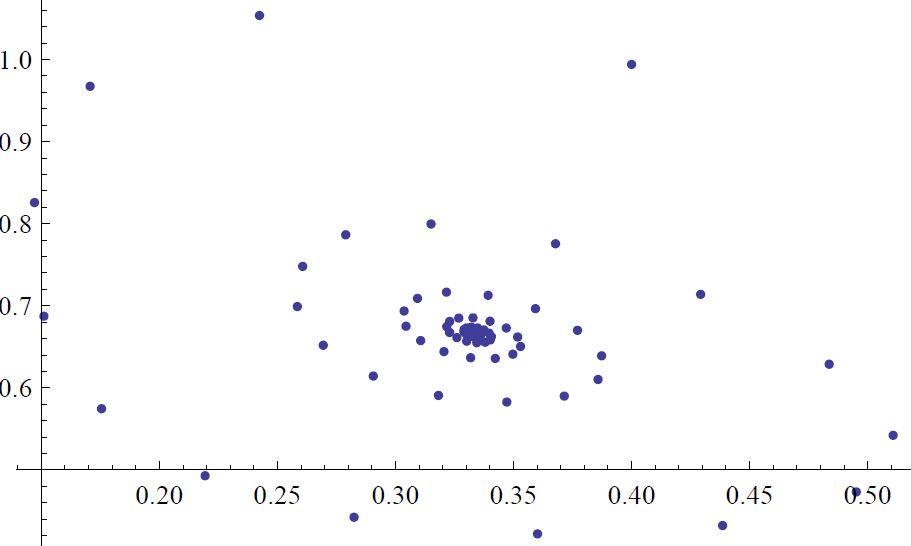}}
    \caption{Two trajectories starting from $x^0=0.8, y^0=0.7.$ }
    \label{dyn}
\end{figure}

\begin{rk} In the cases $r > 1$ or initial point is taken from the set $(x,y)\in \{\mathbb{R}^2_{+}: x > 2, y > 2 - x\}$, the system (\ref{h12}) may produce trajectories with negative coordinates, which are not biologically meaningful.

\end{rk}

\section{Bifurcation analysis}

Proposition \ref{prop1} demonstrates that as $\gamma$ passes through the critical value
$\gamma_0=r+1,$ the dimensions of the stable and unstable manifolds of the fixed point
$E_2$ change, a bifurcation will occur. The following definitions are provided.

Recall that in the dynamical system $(T,X,\phi^t),$ $T$ is a time set, $X$ is a state space and  $\phi^t: X\rightarrow X$ is a family of evolution operators parameterized by $t\in T.$

\begin{defn}(see \cite{Kuz}) A dynamical system $\{T, \mathbb{R}^n, \varphi^t\}$ is called locally topologically equivalent near a fixed point $x_0$ to a dynamical system $\{T, \mathbb{R}^n, \psi^t\}$
near a fixed point  $y_0$ if there exists a homeomorphism $h : \mathbb{R}^n \rightarrow \mathbb{ R}^n$ that
is

(i) defined in a small neighborhood $U\subset \mathbb{R}^n$ of $x_0$;

(ii) satisfies $y_0 = h(x_0)$;

(iii) maps orbits of the first system in $U$ onto orbits of the second system
in $V = f(U)\subset\mathbb{R}^n $, preserving the direction of time.
\end{defn}

Recall that, the phase portrait of a dynamical system is a partitioning of the state space into orbits. In the dynamical system depending on parameters, if parameters vary then the phase portrait also varies. There are two possibilities: either the system remains topologically equivalent to the original one, or its topology changes.

\begin{defn}[see \cite{Kuz}]The appearance of a topologically nonequivalent phase portrait under variation of parameters is called a bifurcation.
\end{defn}

Suppose that given two-dimensional discrete-time system depending on parameters and its Jacobian matrix at the nonhyperbolic fixed point has two complex conjugate eigenvalues $\mu_{1,2}$ with modules one.

\begin{defn}[see \cite{Kuz}]  The bifurcation corresponding to the presence of $\mu_{1,2}$ is called a Neimark-Sacker (or torus) bifurcation.
\end{defn}

Through straightforward calculations, we find that when $\gamma=r+1,$ ($0<r\leq1$) the two complex conjugate eigenvalues of the fixed point  $E_2$ are given by:
\[
\mu_{1,2}=\frac{r+2\mp i\sqrt{3r^2+4r}}{2(r+1)}.
\]
So, the fixed point $E_2$ can pass through a Neimark-Sacker bifurcation when
$0<r\leq1$ and $\gamma$ varies in the small neighborhood of $\gamma_0=1+r.$ To clearly demonstrate the process, we proceed with the following steps.

\textbf{The first step.} Introduce the variable changes $u=x-\hat{x},$ $v=y-\hat{y}$, which shift
the fixed point $E_2=(\hat{x},\hat{y})$ to the origin, and transform system $\ref{h12}$ into
\begin{equation}\label{bif1}
\left\{\begin{aligned}
&u^{(1)}= (u+\hat{x})(1-u-v) -\hat{x}\\
&v^{(1)}= (v+\hat{y})(\gamma u+1) -\hat{y}
\end{aligned}\right.
\end{equation}
where $\hat{x}=\frac{r}{\gamma},$ $\hat{y}=\frac{\gamma-r}{\gamma}.$

\textbf{The second step.} Give a small perturbation $\gamma^*$ to the parameter $\gamma,$ such that $\gamma=\gamma_0+\gamma^*.$ The perturbed system (\ref{bif1}) can then be expressed as:

\begin{equation}\label{bif2}
\left\{\begin{aligned}
&u^{(1)}= (u+\hat{x})(1-u-v) -\hat{x}\\
&v^{(1)}= (v+\hat{y})\left((\gamma_0+\gamma^*)u+1\right) -\hat{y}.
\end{aligned}\right.
\end{equation}

The Jacobian of the system (\ref{bif2}) at the point (0,0) is
\begin{equation}\label{jacob}
J(0,0)=\begin{bmatrix}
1-\hat{x} &~~~ -\hat{x}\\
(\gamma_0+\gamma^*)\hat{y} &~~~ 1
\end{bmatrix}
\end{equation}
and its characteristic equation is
$$\lambda^2-a(\gamma^*)\lambda+b(\gamma^*)=0,$$
where
\[
a(\gamma^*)=Tr(J)=2-\hat{x},
\]
 and
\begin{align*}
 b(\gamma^*)&=\det(J)=1-\hat{x}+\hat{x}\hat{y}(\gamma_0+\gamma^*).
\end{align*}
Note that \(b(\gamma^*) = 1\) when \(\gamma^* = 0\). Therefore, if \(\gamma^*\) varies within a small neighborhood of \(\gamma^* = 0\), we have \(a^2(\gamma^*) - 4b(\gamma^*) < 0\), and the complex roots of the characteristic equation are:
\begin{equation}\label{bif5}
\lambda_{1,2}=\frac{1}{2}[a(\gamma^*)\pm i \sqrt{4b(\gamma^*)-a^2(\gamma^*)}].
\end{equation}
Thus,
\begin{equation}\label{bif6}
|\lambda_{1,2}|=\sqrt{b(\gamma^*)}
\end{equation}
and
\begin{equation}\label{bif7}
\frac{d|\lambda_{1,2}|}{d\gamma^*}\Bigm|_{\gamma^*=0}\!=\! \frac{\hat{x}\hat{y}}{2\sqrt{b(\gamma^*)}}\Bigm|_{\gamma^*=0}\!=\!\frac{\hat{x}\hat{y}}{2}>0.
\end{equation}
Thus, the transversality condition $\left(\frac{d|\lambda_{1,2}|}{d\gamma^*}\Bigm|_{\gamma^*=0}\neq0\right)$ has been proven. Consider the non-degenerate condition, i.e.,  $\lambda_{1,2}^m(0)\neq1$ when $m=1,2,3,4.$ Since $1<a(0)=2-\hat{x}<2$ and $b(0)=1,$ it is easy to deduce that $\lambda_{1,2}^m(0)\neq1$ for all $m=1,2,3,4.$  Hence, all conditions are satisfied for Neimark–Sacker bifurcation to occur.

\textbf{The third step}. Study the normal form of the system (\ref{bif2}) when \(\gamma^* = 0:\)
\begin{equation}\label{bif8}
\left\{\begin{aligned}
&u^{(1)}= (1-\hat{x})u -\hat{x}v-u^2-uv\\
&v^{(1)}= \gamma_0\hat{y}u+v+\gamma_0uv
\end{aligned}\right.
\end{equation}
or
\begin{equation}\label{bif9}
\left\{\begin{aligned}
&u^{(1)}= \frac{u}{r+1}-\frac{rv}{r+1}-u^2-uv\\
&v^{(1)}= u+v+(r+1)uv.
\end{aligned}\right.
\end{equation}
The two eigenvalues of the linear part of the system (\ref{bif8}) are:
$$\lambda_{1,2}=\frac{2-\hat{x}\mp i\sqrt{4\hat{x}-\hat{x}^2}}{2}=\frac{r+2\mp i\sqrt{3r^2+4r}}{2(r+1)},$$
where $\hat{x}=\frac{r}{\gamma_0}=\frac{r}{r+1}.$
The corresponding eigenvectors are
$$\overrightarrow{v}_{1,2}=\begin{bmatrix}-r\\2(r+1)\end{bmatrix}\mp i\begin{bmatrix}\sqrt{3r^2+4r}\\ 0\end{bmatrix}.$$

\textbf{The fourth step}. We find the normal form of the system (\ref{bif2}). Let matrix
\[
M= \begin{bmatrix} \sqrt{3r^2+4r} & -r\\ 0 & 2(r+1)\end{bmatrix}
\]
then

\[
M^{-1}= \begin{bmatrix} \frac{1}{\sqrt{3r^2+4r}} & \frac{r}{2(r+1)\sqrt{3r^2+4r}} \\
0 & \frac{1}{2(r+1)}\end{bmatrix}.
\]

By transformation, we get that
\begin{equation}\label{bif10}
(u,v)^T=M\cdot(X,Y)^T
\end{equation}
and the system (\ref{bif8}) transforms into the following system

\begin{equation}\label{bif9}
\left\{\begin{aligned}
&X^{(1)}= \frac{r+2}{2(r+1)}X-\frac{\sqrt{3r^2+4r}}{2(r+1)}Y+F(X,Y)\\
&Y^{(1)}= \frac{\sqrt{3r^2+4r}}{2(r+1)}X+\frac{r+2}{2(r+1)}Y+G(X,Y),
\end{aligned}\right.
\end{equation}
where
\begin{equation}
\begin{split}
&F(X,Y)=-\sqrt{3r^2+4r}X^2+(r^2+r-2)XY+\frac{2r-r^3}{\sqrt{3r^2+4r}}Y^2,\\
&G(X,Y)=(r+1)\sqrt{3r^2+4r}XY-r(r+1)Y^2.
\end{split}
\end{equation}

In addition, the partial derivatives at $(0,0)$ are
\begin{equation}
\begin{split}
&F_{XX}=-2\sqrt{3r^2+4r}, \ \ F_{XY}=r^2+r-2, \ \ F_{YY}=\frac{4r-2r^3}{\sqrt{3r^2+4r}}, \\
&F_{XXX}=F_{XYY}=F_{XXY}=F_{YYY}=0,\\
&G_{XX}=0, \ \ G_{XY}=(r+1)\sqrt{3r^2+4r}, \ \ G_{YY}=-2r(r+1), \\
&G_{XXX}=G_{XXY}=G_{XYY}=G_{YYY}=0.
\end{split}
\end{equation}

\textbf{The fifth step}. We need to compute the discriminating quantity $\mathcal{L}$ via the following formula (see \cite{Rob}), which determines
the stability of the invariant circle bifurcated from Neimark-Sacker bifurcation of the system (\ref{bif9}):
\begin{equation}\label{lya}
\mathcal{L}=-Re\left[\frac{(1-2\lambda_1)\lambda_2^2}{1-\lambda_1}L_{11}L_{20}\right]-\frac{1}{2}|L_{11}|^2-|L_{02}|^2+Re(\lambda_2 L_{21}),
\end{equation}
where
\begin{equation}
\begin{split}
&L_{20}=\frac{1}{8}[(F_{XX}-F_{YY}+2G_{XY})+i(G_{XX}-G_{YY}-2F_{XY})],\\
&L_{11}=\frac{1}{4}[(F_{XX}+F_{YY})+i(G_{XX}+G_{YY})],\\
&L_{02}=\frac{1}{8}[(F_{XX}-F_{YY}-2G_{XY})+i(G_{XX}-G_{YY}+2F_{XY})],\\
&L_{21}=\frac{1}{16}[(F_{XXX}\!+\!F_{XYY}\!+\!G_{XXY}\!+\!G_{YYY})\!+\!i(G_{XXX}\!+\!G_{XYY}\!-\!F_{XXY}\!-\!F_{YYY})].
\end{split}
\end{equation}

After some calculations, we obtain:
\begin{align}
&L_{20}=\frac{r^3+r^2-0.5r}{\sqrt{3r^2+4r}}+\frac{i}{2},  \notag\\
&L_{11}=-\frac{r^3+3r^2+2r}{2\sqrt{3r^2+4r}}-\frac{r(r+1)}{2}i, \notag\\
&L_{02}=-\frac{r^3+5r^2+5r}{2\sqrt{3r^2+4r}}+\frac{r^2+r-1}{2}i,\\
&L_{21}=0. \notag
\end{align}

Thus, we get
\begin{align}
&|L_{11}|^2=\frac{r(r^4+4r^3+8r^2+4r+1)}{(3r+4)} \notag\\
&|L_{02}|^2=\frac{r^5+5r^4+10r^3+10r^2+5r+1}{(3r+4)} \notag\\
&L_{11}L_{20}=-\frac{r(r+1)(r^3+3r^2-3)}{2(3r+4)}-\frac{r(r+1)(r^3+r^2+1)}{2\sqrt{3r^2+4r}}i \notag\\
&\frac{(1-2\lambda_1)\lambda_2^2}{1-\lambda_1}=-\frac{3r^2+5r+1}{2(r+1)^2}+\frac{(r^2+3r+1)\sqrt{3r^2+4r}}{2r(r+1)^2}i
\end{align}
After performing some calculations, we find that:
\[
\mathcal{L}=\mathcal{L}(r)=-\frac{6r^6+32r^5+64r^4+60r^3+36r^2+19r+4}{2(r+1)(3r+4)}<0 \ \ (0<r\leq1).
\]
 In summary, the discussions above lead us to the following concluding theorem.

\begin{thm}\label{bifurcation} Assume that $0 < r \leq 1$, $\gamma > r$, and $\gamma_0 = 1 + r$. If the parameter $\gamma$ varies within a small neighborhood of $\gamma_0$, then the system (\ref{h12}) undergoes a Neimark-Sacker bifurcation at the fixed point $E_2 = (\hat{x}, \hat{y})$. Furthermore, since $\mathcal{L} < 0$, an attracting invariant closed curve bifurcates from the fixed point for $\gamma > \gamma_0$.
\end{thm}

\textbf{Example. } Consider the system (\ref{h12}) with parameters $r = 0.5$ and $\gamma = \gamma_0 = 1.5$. The fixed point is $E_2 = \left(\frac{1}{3}, \frac{2}{3}\right)$ with multipliers $\lambda_1 = \frac{5 - i\sqrt{11}}{6}$ and $\lambda_2 = \frac{5 + i\sqrt{11}}{6}$. Moreover, $|\lambda_{1,2}| = 1$, and
$
\frac{d|\lambda_{1,2}|}{d\gamma^*} \bigg|_{\gamma^* = 0} = \frac{1}{9} > 0,
$
while $\mathcal{L} \approx -2.1 < 0$. Hence, according to Theorem \ref{bifurcation}, an attracting invariant closed curve bifurcates from the fixed point for $\gamma^* > 0$. (Fig. \ref{bifur})

\begin{figure}[h!]
    \centering
    \subfigure[\tiny$x^0=0.3, y^0=0.7$]{\includegraphics[width=0.45\textwidth]{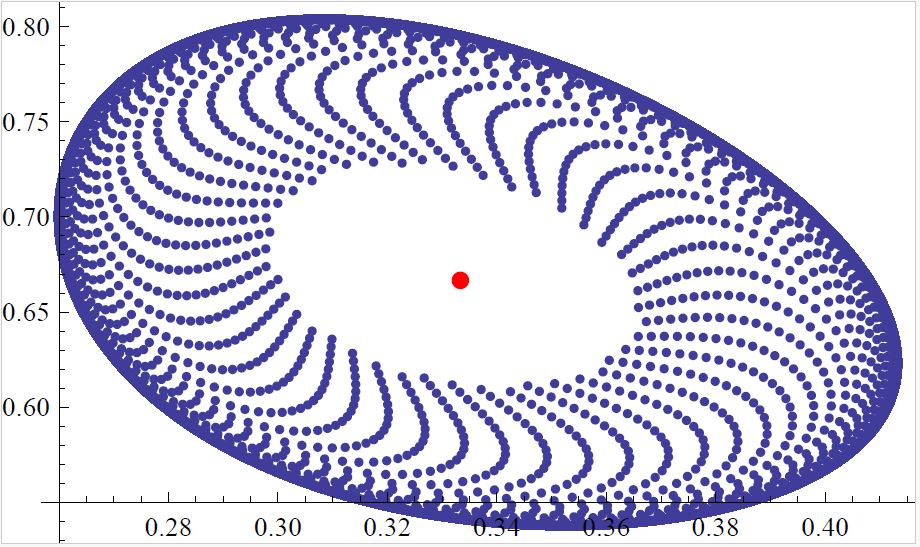}} \hspace{0.3in}
    \subfigure[\tiny $x^0=0.8, y^0=1$]{\includegraphics[width=0.428\textwidth]{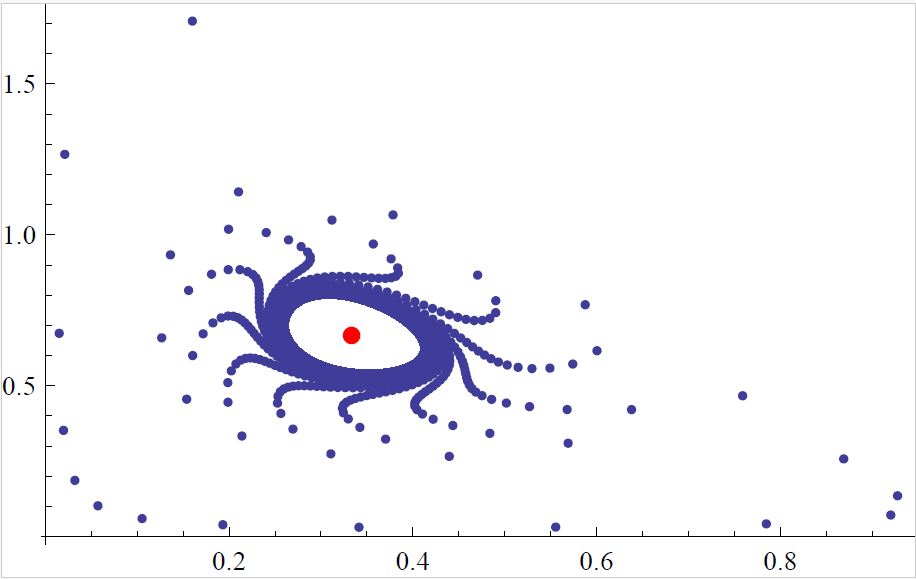}}
    \caption{Phase portraits for the system (\ref{h12}) with $r=0.5, \gamma=1.51$ i.e., $\gamma^*=0.01.$ }
    \label{bifur}
\end{figure}

\section*{Conclusion}
In this paper, we thoroughly investigate both the continuous and discrete-time dynamics of a phytoplankton-zooplankton model with linear functional responses.

For the continuous-time case, it is shown that if \( 0 < \gamma \leq r \), the equilibrium point \( E_1 = (1, 0) \) is globally attracting. If \( \gamma > r \), then the positive equilibrium point \( E_2 \) becomes globally attracting. These results are established by constructing a Lyapunov function (Theorem~\ref{thm1} and Theorem~\ref{thm2}).

To analyze the dynamics of the discrete analogue of the model, defined by the operator (\ref{h12}), we identify a closed invariant set (Lemma~\ref{lem2}) and apply LaSalle's Invariance Principle to prove the global attractivity of the equilibrium points (Theorem~\ref{thm3}).

Moreover, we demonstrate the occurrence of a Neimark–Sacker bifurcation at the positive equilibrium point (Theorem~\ref{bifurcation}), indicating that the discrete model exhibits rich dynamical behavior. Additionally, we show that a closed invariant curve is always attracting by proving that the discriminating quantity \( \mathcal{L}\), which is often difficult to verify, is always negative. The occurrence of a bifurcation at a positive fixed point plays a crucial role in understanding the qualitative behavior of ecological systems such as the phytoplankton-zooplankton interaction model. It marks a critical threshold where the system's dynamics can shift dramatically -- from stable coexistence to oscillatory or unstable regimes. Such bifurcations highlight the sensitivity of the ecosystem to parameter changes and can reveal tipping points beyond which the persistence of species may no longer be guaranteed. Recognizing and analyzing these bifurcations not only enhances our theoretical understanding but also provides valuable insights for managing and predicting real-world ecological outcomes.

In contrast to the continuous-time case, the positive equilibrium point \( E_2 \) is not necessarily globally attracting in the discrete case. It becomes globally attracting only when the parameters belong to a specific parameter set \( S \).


\end{document}